\documentclass[reqno,12pt]{amsart}%
\usepackage{amsfonts}
\usepackage{amsmath}
\usepackage{amssymb}
\usepackage{graphicx}%
\def\mbE{{\mathbb{E}}}
\def\mbF{{\mathbb{F}}}

\def\mbR{{\mathbb{R}}}

\def\bA{{\mathbf{A}}}

\def\bM{{\mathbf{M}}}

\def\cE{{\mathcal{E}}}

\def\cJ{{\mathcal{J}}}

\def\cU{{\mathcal{U}}}
\def\cV{{\mathcal{V}}}
\def\cL{{\mathcal{L}}}

\def\mfku{{\mathfrak{u}}}

\def\inc{u^0}

\def\oh{\hat{\otimes}}

\def\a{\alpha}
\def\dW{{\dot{W}}}

\def\ds{\displaystyle}

\newtheorem{theorem}{Theorem}
\theoremstyle{plain}
\newtheorem{corollary}[theorem]{Corollary}
\newtheorem{definition}[theorem]{Definition}
\newtheorem{proposition}[theorem]{Proposition}
\newtheorem{example}[theorem]{Example}

\newtheorem{remark}[theorem]{Remark}

\renewcommand\endproof{\hfill $\Box$\vskip 0.15in}

\numberwithin{equation}{section}
\numberwithin{theorem}{section}

\begin{document}
\title[Skorokhod Integral and SPDEs]
{ A Unified Approach to
Stochastic Evolution Equations Using the Skorokhod Integral}
\author{S. V. Lototsky}
\curraddr[S. V. Lototsky]{Department of Mathematics, USC\\
Los Angeles, CA 90089}
\email[S. V. Lototsky]{lototsky@math.usc.edu}
\urladdr{http://math.usc.edu/$\sim$lototsky}
\author{B. L. Rozovskii}
\curraddr[B. L. Rozovskii]{Division of Applied Mathematics \\
Brown University\\
Providence, RI 02912}
\email[B. L. Rozovskii]{rozovsky@dam.brown.edu}
\urladdr{http://www.dam.brown.edu/people/rozovsky.html}
\thanks{S. V. Lototsky acknowledges support
 from  the
NSF CAREER award DMS-0237724.}
\thanks{B. L. Rozovskii acknowledges support
 from NSF Grant DMS
0604863, ARO Grant W911NF-07-1-0044, and
ONR Grant N00014-07-1-0044.}

\begin{abstract}
We study stochastic evolution equations
 driven by Gaussian
noise. The key features of the model are that the operators
in the deterministic and stochastic parts can have the same order
and the noise can be time-only, space-only, or space-time.
   Even the simplest equations of this kind do
not have a
square-integrable solution and must be solved in special weighted
spaces. We
demonstrate that the Cameron-Martin version of the Wiener chaos
decomposition leads to natural weights and a natural
replacement of the square integrability condition.

\end{abstract}
\maketitle

\today

\section{Introduction}

Let $W$ be a standard Brownian motion, $u_0\in L_2(\mbR)$, a
deterministic function, and $\sigma$, a real number.
It is well known  that
the It\^{o} equation
\begin{equation}
\label{intr0}
u(t,x)=u_0(x)+\int_0^tu_{xx}(s,x)ds+\int_0^t\sigma u(s,x)dW(s)
\end{equation}
 has a solution $u(t,x)$ and
 \begin{equation}
 \label{intr1}
 \int_{\mbR}\mbE u^2(t,x)dx<\infty
 \end{equation}
 for every $t>0$ and all $\sigma\in \mbR$;
 see, for example, \cite{Roz} or simply
consider the Fourier transform of the equation.
 On the other hand,  the solution of equation
 \begin{equation}
\label{intr2}
u(t,x)=u_0(x)+\int_0^tu_{xx}(s,x)ds+\int_0^t\sigma u_x(s,x)dW(s)
\end{equation}
satisfies \eqref{intr1} for an arbitrary $u_0\in L_2(\mbR)$
if and only if $|\sigma|\leq \sqrt{2}$,
while equation
\begin{equation}
\label{intr3}
u(t,x)=u_0(x)+\int_0^tu_{xx}(s,x)ds+\int_0^t\sigma u_{xx}(s,x)dW(s)
\end{equation}
hardly ever has a solution satisfying \eqref{intr1}.
Still, because of the It\^{o} integral, all three equations
share a common feature, namely, that the mean evolution
is described by the heat equation. One of the objectives of this work
is to develop a unified approach that would cover all three equations
above without any special restrictions on the initial condition or
the amplitude of noise.

The situation becomes even more complicated when the noise
has a spacial component. As an illustration, take a standard
random variable $\xi$ and consider the equation
\begin{equation*}
\label{intr333}
u(t)=1+\int_0^tu(s)ds+\int_0^tu(s)\xi ds.
\end{equation*}
The solution is clearly $u(t)=e^te^{\xi t}$. The unperturbed
solution is $e^t$, so that
$\mbE  u(t)\not=e^t$.

It turns out that, to preserve the
mean dynamics, the usual product $u\xi$ must be replaced by
the Wick product $u\diamond \xi$. In particular, the solution of
\begin{equation}
\label{intr4}
u(t)=1+\int_0^t u(s)ds+\int_0^t u(s)\diamond \xi ds
\end{equation}
is $u(t)=e^te^{\xi t- (t^2/2)}$ and satisfies $\mbE  u(t)=e^t$.
For partial differential equations of this
type, such as
\begin{equation}
\label{intr5}
u(t,x)=u_0(x)+\int_0^tu_{xx}(s,x)ds+\int_0^tu_x(s,x)\diamond \xi ds,
\end{equation}
inequality  \eqref{intr1}, without additional assumptions on $u_0$,
  holds only for sufficiently small $t$.
Still, the same framework that unifies equations
\eqref{intr0}, \eqref{intr2}, \eqref{intr3},
also covers \eqref{intr4}, \eqref{intr5},  and a
wide class of other equations with space and space-time
 Gaussian noise.

 There are  three main components in the general construction
 presented in this paper:
 \begin{enumerate}
 \item Use of the Skorokhod integral $\boldsymbol{\delta}$
  for all types of random perturbation.
 \item Separation of the time as the variable in the equation
 from the possible time evolution of the noise.
 \item Replacement of the space $L_2(\Omega; X)$ of square-integrable
 $X$-valued processes with a space $(\cL)_{Q,r}(X)$ of
 generalized (and not necessarily adapted)
  random elements as the space for both the input
 data and the solution.
\end{enumerate}

Section \ref{sec2} presents the necessary definitions and constructions
related to the Skorokhod integral with respect to an abstract Gaussian
white noise and Section \ref{sec3} introduces the spaces
$(\cL)_{Q,r}$. Section \ref{sec4} presents the unified treatment
of stochastic evolution equations.
While the main results, Theorem \ref{th:main} cannot
be fully understood without all the preliminaries, there is
a certain analogy with the corresponding result for the
deterministic equations. Namely, let $(V,H,V')$ be a
normal triple of Hilbert spaces (defined precisely in
Section \ref{sec4}) and let $\bA(t)$, $t\in [0,T]$ be a
family of  bounded linear operators from $V$ to $V'$.
It is known from deterministic analysis, that
if  the family of operators $\bA(t)$ is uniformly elliptic,
then the initial value problem for the equation
$u(t)=u_0+\int_0^t (\bA(s)u(s)+f(s))ds$ is well-posed in the
normal triple $(V,H,V')$.

Now consider the initial value problem for the stochastic equation
\begin{equation}
\label{intr11}
u(t)=u_0+\int_0^t(\bA(s)u(s)+f(s))ds+\int_0^t\boldsymbol{\delta}
(\bM(s)u(s))ds.
\end{equation}
 The statement of Theorem \ref{th:main} is
essentially as follows: if the family of operators
$\bA(t)$ is uniformly elliptic and $\bM(t)$ is a
family of continuous operators from $V$ to $V'$, then
the initial value problem \eqref{intr11} is well
posed in the normal triple
$((\cL)_{Q,r}(V),(\cL)_{Q,r}(H), (\cL)_{Q,r}(V'))$.
 The reader should not miss Example \ref{example:main}
 illustrating how many familiar equations are particular cases
 of \eqref{intr11}.

\section{Skorokhod Integral With Respect to Gaussian White Noise}

\label{sec2}

Let ${\mathbb{F}}=(\Omega,{\mathcal{F}},\mathbb{P})$ be a complete
probability space, and ${\mathcal{U}}$, a real separable Hilbert
space with inner product $(\cdot,\cdot)_{{\mathcal{U}}}$. On
${\mathbb{F}}$, consider a zero-mean Gaussian family
\[
{\dot{W}}=\left\{  {\dot{W}}(h),\ h\in{\mathcal{U}}\right\}
\]
so that
\[
{\mathbb{E}}\left(  {\dot{W}}(h_{1})\;{\dot{W}}(h_{2})\right)  =(h_{1}%
,h_{2})_{{\mathcal{U}}}.
\]
It suffices, for our purposes, to assume that ${\mathcal{F}}$ is the
$\sigma $-algebra generated by ${\dot{W}(h), \, h\in \mathcal{U}}$.
Given a real separable Hilbert space $X$, we denote by
$L_{2}({\mathbb{F}};X)$ the Hilbert space of square-integrable
${\mathcal{F}}$-measurable $X$-valued random elements $f$. In
particular,
\[
(f,g)_{L_{2}({\mathbb{F}};X)}^{2}:={\mathbb{E}}(f,g)_{X}^{2}.
\]
When $X={\mathbb{R}}$, we write $L_{2}({\mathbb{F}})$ instead of
$L_{2}({\mathbb{F}};{\mathbb{R}})$.

\begin{definition}
\label{def:well} A formal series
\begin{equation}
\dot{W}=\sum_{k\geq 1}\dot{W}({\mathfrak{u}}_{k})\, {\mathfrak{u}}_{k} ,
\label{eq:wn}%
\end{equation}
where $\left\{  {\mathfrak{u}}_{k},k\geq1\right\}  $ is a complete
orthonormal basis in $\mathcal{U}$,\ is called Gaussian
\textbf{white noise} on $\mathcal{U}.$
\end{definition}

\begin{example}
If $\cU=L_2((0,T))$, then $\dot{W}(h)=\int_0^T h(s)dw(s)$,
where $w(t)=\dot{W}(\chi_{{}_t})$ is the standard Brownian motion;
 $\chi_{{}_t}$ is the characteristic function of the interval
 $[0,t]$.
 \end{example}

Given an orthonormal basis
${\mathfrak{U}}=\{{\mathfrak{u}}_{k},k\geq1\}$ in ${\mathcal{U}}$,
define a collection $\{\xi_{k},k\geq1\}$ of independent
standard Gaussian random variables by
$\xi_{k}={\dot{W}}({\mathfrak{u}%
}_{k})$. Let  ${\mathcal{J}}$ be the collection of multi-indices
$\alpha$ with $\alpha=(\alpha_{1},\alpha_{2},\ldots)$ so that each
$\alpha_{k}$ is a non-negative integer and
$|\alpha|:=\sum_{k\geq1}\alpha_{k}<\infty$. For
$\alpha,\beta\in{\mathcal{J}}$, we define
\[
\alpha+\beta=(\alpha_{1}+\beta_{1},\alpha_{2}+\beta_{2},\ldots),\quad
\alpha!=\prod_{k\geq1}\alpha_{k}!.
\]
By $(0)$ we denote the multi-index with all zeroes. By
$\varepsilon_{i}$ we denote the multi-index $\alpha$ with
$\alpha_{i}=1$ and $\alpha_{j}=0$ for $j\not =i$. With this
notation, $n\varepsilon_{i}$ is the multi-index $\alpha$ with
$\alpha_{i}=n$ and $\alpha_{j}=0$ for $j\not =i$.

Define the collection of random variables $\Xi=\{\xi_{\alpha},
\alpha \in{\mathcal{J}}\}$ as follows:
\begin{equation}
\label{eq:basis}\xi_{\alpha} = \prod_{k} \left(
\frac{H_{\alpha_{k}}(\xi _{k})}{\sqrt{\alpha_{k}!}} \right)  ,
\end{equation}
where
\begin{equation}
\label{eq:hermite}
H_{n}(x) =
(-1)^{n} e^{x^{2}/2}\frac{d^{n}}{dx^{n}}%
e^{-x^{2}/2}%
\end{equation}
is Hermite polynomial of order $n$.

\begin{theorem}[Cameron and Martin \cite{CM}]
\label{th:CM}  The collection
$\displaystyle\Xi =\{\xi_{\alpha},\ \alpha\ \in{\mathcal{J}}\}$ is
an orthonormal basis in
$\displaystyle L_{2}({\mathbb{F}};X)$: if $\displaystyle\eta\in L_{2}%
({\mathbb{F}};X)$ and
$\displaystyle\eta_{\alpha}={\mathbb{E}}(\eta\xi_{\alpha })$, then
$\displaystyle\eta=
\sum_{\alpha\in{\mathcal{J}}}\eta_{\alpha}%
\xi_{\alpha}$ and ${\mathbb{E}}|\eta|^{2}
=\sum_{\alpha\in{\mathcal{J}}}%
\eta_{\alpha}^{2}.$
\end{theorem}

Denote by ${\mathbf{D}}$ the \emph{Malliavin derivative} on $L_{2}%
({\mathbb{F}})$ (see e.g. \cite{Nualart}). In particular,
 if $F:{\mathbb{R}%
}^{N}\rightarrow{\mathbb{R}}$ is a smooth function and
 $h_{i}\in{\mathcal{U}%
},\ i=1,\ldots N$, then
\begin{equation}
{\mathbf{D}}F({\dot{W}}(h_{1}),\ldots{\dot{W}}(h_{N}))=\sum_{i=1}^{N}%
\frac{\partial F}{\partial x_{i}}({\dot{W}}(h_{1}),\ldots,{\dot{W}}%
(h_{N}))h_{i}\in L_{2}({\mathbb{F}};{\mathcal{U}}). \label{eq:MD}%
\end{equation}
It is known \cite{Nualart} that the domain
${\mathbb{D}}^{1,2}({\mathbb{F}})$
of the operator ${\mathbf{D}}$ is a dense linear subspace of $L_{2}%
({\mathbb{F}})$.

The adjoint of the Malliavin derivative on $L_2(\mathbb{F})$ is the
{\tt Skorokhod integral} and is traditionally denoted by
$\boldsymbol{\delta}$
\cite{Nualart}. The operator $\boldsymbol{\delta}$
extends to a subspace ${\mathbb{D}}^{1,2}({\mathbb{F}};X\otimes \mathcal{U})$
 of  $L_2(\mathbb{F}; X\otimes \mathcal{U})$
for a Hilbert space $X$: given
$f\in {\mathbb{D}}^{1,2}({\mathbb{F}};X\otimes \mathcal{U})$,
$\boldsymbol{\delta}(f)$ is the unique element of $L_2(\mathbb{F};X)$
with the property
\begin{equation}
\mathbb{E}(\varphi{\boldsymbol{\delta}}(f))=
{\mathbb{E}}(f,{\mathbf{D}}\varphi)_{{\mathcal{U}}}
\label{eq:SkI}%
\end{equation}
for all $\varphi\in {\mathbb{D}}^{1,2}({\mathbb{F}})$.

We will now derive the expressions for the operators $\mathbf{D}$
and  $\boldsymbol{\delta}$  in the basis $\Xi.$
 To begin, let us  compute ${\mathbf{D}}(\xi_{\alpha}).$

\begin{proposition}
\label{prop:Dxi} For each $\alpha\in{\mathcal{J}},$
\begin{equation}
{\mathbf{D}}(\xi_{\alpha})
=\sum_{k\geq1}\sqrt{\alpha_{k}}\,\xi_{\alpha
-\varepsilon_{_{k}}}{\mathfrak{u}}_{k}. \label{eq:Dxi}%
\end{equation}

\end{proposition}

\textbf{Proof.} The result follows by direct computation using the
property (\ref{eq:MD}) of the Malliavin derivative and the relation
$H_{n}^{\prime }(x)=nH_{n-1}(x)$ for the Hermite polynomials (cf.
\cite[Chapter 1]{Nualart}).
\endproof

\begin{remark}
The set $\mathcal{J}$ is not closed under
substraction and  the expression $\alpha-\varepsilon_{k}$
is undefined if $\alpha_{k}=0$. Everywhere in this paper when
undefined expressions of this type appear, we use the following
convention: if $\alpha_{k}=0,$ then $\sqrt{\alpha_{k}}\,\xi_{\alpha
-\varepsilon_{k}}=0$.
\end{remark}

\begin{proposition}
\label{prop:Ixi} For $\xi_{\alpha}\in\Xi$, $h\in X$,
and ${\mathfrak{u}}_{k}\in{\mathfrak{U}}$,
\begin{equation}
{\boldsymbol{\delta}}(\xi_{\alpha}\, h\otimes{\mathfrak{u}}_{k})
 =h\,\sqrt{\alpha_{k}+1}\,\xi_{\alpha+\varepsilon_{k}}.
\label{eq:Ixi}%
\end{equation}

\end{proposition}

\textbf{Proof.} It is enough to verify (\ref{eq:SkI}) with
$f=h\otimes {\mathfrak{u}}_{k}\,\xi_{\alpha}$ and
$\varphi=\xi_{\beta}$, where $h\in X$. By (\ref{eq:Dxi}),
\[
{\mathbb{E}}(f,{\mathbf{D}}\varphi)_{{\mathcal{U}}}=
\sqrt{\beta_{k}}\, h\,
{\mathbb{E}}(\xi_{\alpha}\xi_{\beta-\varepsilon_{k}}) =
\begin{cases}
\sqrt{\alpha_{k} +1}\, h, & \mathrm{if}\ \alpha=\beta-\varepsilon_{k},\\
0, & \mathrm{if } \ \alpha\not =\beta-\varepsilon_{k}.
\end{cases}
\]
In other words,
\[
{\mathbb{E}}(\xi_{\alpha}\,
h\otimes{\mathfrak{u}}_{k},{\mathbf{D}}\xi_{\beta
})_{{\mathcal{U}}}=h\,{\mathbb{E}}(\sqrt{\alpha_{k}+1}\xi_{\alpha
+\varepsilon_{k}}\xi_{\beta})
\]
for all $\beta\in{\mathcal{J}}$. \endproof

\begin{remark} The operator $\boldsymbol{\delta}\mathbf{D}$
is linear and unbounded on $L_2(\mathbb{F})$;
 it follows from Propositions
\ref{prop:Dxi} and \ref{prop:Ixi} that the random variables
$\xi_{\alpha}$ are eigenfunctions of this operator:
\begin{equation}
\label{eq:eig-xi} \boldsymbol{\delta}(\mathbf{D}(\xi_{\alpha}))
=|\alpha|\xi_{\alpha}.
\end{equation}
\end{remark}

To give an alternative characterization of the operator
 ${\boldsymbol{\delta}}$,
we define a new operation on the elements of $\Xi$.

\begin{definition}
\label{def:WP} For $\xi_{\alpha}$, $\xi_{\beta}$ from $\Xi$, define
the Wick product
\begin{equation}
\label{eq:WP}
\xi_{\alpha}\diamond\xi_{\beta}:=\sqrt{\left(  \frac
{(\alpha+\beta)!}{\alpha!\beta!} \right)  } \xi_{\alpha+\beta}.
\end{equation}

\end{definition}

In particular, taking in (\ref{def:WP}) $\alpha=k\varepsilon_{i}$
and $\beta=n\varepsilon_{i}$, and using (\ref{eq:basis}), we get
\begin{equation}
\label{eq:WP-hp}H_{k}(\xi_{i})\diamond
H_{n}(\xi_{i})=H_{k+n}(\xi_{i}).
\end{equation}
An immediate consequence of Proposition \ref{prop:Ixi} and
Definition \ref{def:WP} is the following identity:
\begin{equation}
\label{eq:IWP}{\boldsymbol{\delta}}
(\xi_{\alpha}h\otimes{\mathfrak{u}}_{k}%
)=h\xi_{\alpha}\diamond\xi_{k}, \ h\in X.
\end{equation}

More generally, we define the operation $\diamond$ on formal
series:
 $$
 \left(\sum_{\alpha\in \cJ} f_{\alpha}\xi_{\alpha}\right)
 \diamond \left(\sum_{\alpha\in \cJ} g_{\alpha} \xi_{\alpha}
 \right)=
 \sum_{\alpha\in \cJ} \left(
 \sum_{\beta,\gamma\in \cJ:\beta+\gamma=\alpha}
 \sqrt{\frac{\alpha!}{\beta!\gamma!}}f_{\beta}g_{\gamma}
 \right)\xi_{\alpha}
 $$
 for $f_{\alpha}\in X$, $g_{\alpha}\in \mbR$.

 \begin{theorem}
\label{th:dual}
If
$f=\sum_{k\geq1}f_{k}\otimes{\mathfrak{u}%
}_{k}$,
$f_{k}=\sum_{\alpha\in{\mathcal{J}}}f_{k,\alpha}\,\xi_{\alpha}%
\in{\mathcal{R}}L_{2}({\mathbb{F}};X)$, and
$f$ is in the domain of $\boldsymbol{\delta}$, then
\begin{equation}
{\boldsymbol{\delta}}(f)=
\sum_{k\geq1}f_{k}\diamond\xi_{k}, \label{eq:IWPG}%
\end{equation}
and
\begin{equation}
({\boldsymbol{\delta}}(f))_{\alpha}=
\sum_{k\geq1}\sqrt{\alpha_{k}}f_{k,\alpha
-\varepsilon_{k}}.
\label{eq:IBasG}%
\end{equation}

\end{theorem}

\begin{proof}
 By linearity and (\ref{eq:IWP}),
\[
{\boldsymbol{\delta}}(f)=\sum_{k\geq1}
\sum_{\alpha\in{\mathcal{J}}} {\boldsymbol{\delta}}%
(\xi_{\alpha} f_{k,\alpha}\otimes{\mathfrak{u}}_{k})= \sum_{k\geq1}
\sum_{\alpha\in{\mathcal{J}}}
f_{k,\alpha}\xi_{\alpha}\diamond\xi_{k}= \sum_{k\geq1}
f_{k}\diamond\xi_{k},
\]
which is (\ref{eq:IWPG}). On the other hand, by (\ref{eq:Ixi}),
\[
{\boldsymbol{\delta}}(f)=\sum_{k\geq1}
\sum_{\alpha\in{\mathcal{J}}} f_{k,\alpha}%
\sqrt{\alpha_{k}+1}\, \xi_{\alpha+\varepsilon_{k}}=
\sum_{k\geq1}\sum _{\alpha\in{\mathcal{J}}}
f_{k,\alpha-\varepsilon_{k}}\sqrt{\alpha_{k}}\, \xi_{\alpha},
\]
and (\ref{eq:IBasG}) follows.
\end{proof}

To proceed, we need a description of a multi-index $\alpha$
with $|\alpha|=n>0$ using its characteristic set $K_{\alpha}$, that
is, an ordered $n$-tuple $K_{\alpha }=\{k_{1},\ldots,k_{n}\}$, where
$k_{1}\leq k_{2}\leq\ldots\leq k_{n}$ characterize the locations and
the values of the non-zero elements of $\alpha $. More precisely,
$k_{1}$ is the index of the first non-zero element of $\alpha,$
followed by $\max\left( 0,\alpha_{k_{1}}-1\right)  $ of entries with
the same value. The next entry after that is the index of the second
non-zero element of $\alpha$, followed by $\max\left(  0,\alpha_{k_{2}%
}-1\right)  $ of entries with the same value, and so on.\ For
example, if $n=7$ and $\alpha=(1,0,2,0,0,1,0,3,0,\ldots)$, then the
non-zero elements of $\alpha$ are $\alpha_{1}=1$, $\alpha_{3}=2$,
$\alpha_{6}=1$, $\alpha_{8}=3$.
As a result, $K_{\alpha}=\{1,3,3,6,8,8,8\}$, that is, $k_{1}=1,\,k_{2}%
=k_{3}=3,\,k_{4}=6, k_{5}=k_{6}=k_{7}=8$. Note  also that, for every
sequence $(b_k,k\geq 1)$ of positive numbers,
\begin{equation}
\prod_{k\geq 1}b_k^{\alpha_k}= b_{i_1}\cdot b_{i_2}\cdot \ldots
\cdot b_{i_n}=\prod_{i\in K_{\alpha}} b_i,
\end{equation}
which can serve as  an equivalent definition of $K_{\alpha}$.

Using the notion of the characteristic set, we now state the
following analog of the well-known result of It\^{o} \cite{Ito}
connecting multiple Wiener integrals and Hermite polynomials.

\begin{proposition}
\label{rem:xial} Let $\alpha\in{\mathcal{J}}$ be a multi-index with
$|\alpha|=n\geq1$ and characteristic set $K_{\alpha}=\{k_{1},\ldots,
k_{n}\}$. Then
\begin{equation}
\label{xial-alt}
\xi_{\alpha}= \frac{\xi_{k_{1}}\diamond\xi_{k_{2}}%
\diamond\cdots\diamond\xi_{k_{n}}}{\sqrt{\alpha!}}.
\end{equation}

\end{proposition}

\textbf{Proof.} This follows from (\ref{eq:basis}) and
(\ref{eq:WP-hp}), because by (\ref{eq:WP-hp}), for every $i$ and
$k$,
\[
H_{k}(\xi_{i})=\underset{k\,\mathrm{times}}
{\underbrace{\xi_{i}\diamond
\cdots\diamond\xi_{i}}}.
\]
\endproof

 By induction, we define the operator
$\boldsymbol{\delta}^{\oh n}$, $n> 1$ on the space
$L_2(\mathbb{F}; X\otimes
\mathcal{U}^{\oh n})$, where $\mathcal{U}^{\oh n}$ is the symmetric
tensor power of $\mathcal{U}$. Then relation \eqref{eq:IWP} leads
to an alternative form of \eqref{xial-alt}:
\begin{equation}
\label{mult-int}
 \xi_{\alpha}=\frac{1}{|\alpha|!\,\sqrt{\alpha!}}\,
\boldsymbol{\delta}^{\oh |\alpha|}(E_{\alpha}),
\end{equation}
where
\begin{equation}
\label{Ealpha}
E_{\alpha}=\sum_{\sigma\in \mathcal{P}}
\mfku_{i_{\sigma(1)}}\otimes\cdots\otimes\mfku_{i_{\sigma(n)}},
\end{equation}
with summation in \eqref{Ealpha} taken over all permutations
$\mathcal{P}$ of the set $\{1,\ldots,n\}$. This leads to the
following generalization of the multiple Wiener integral
expansion for the elements of $L_2(\mathbb{F};X)$.

\begin{theorem}
\label{th:mult-int}
If $\eta\in L_2(\mathbb{F};X)$, then there is a unique
collection of the deterministic elements $\eta_k$, $k\geq 0$,
with $\eta_0=\mathbb{E}\eta\in X$ and
$\eta_k\in X\otimes \mathcal{U}^{\oh k}$, $k\geq 1$,
with the properties
\begin{equation}
\label{eq:mult-int}
\eta=\eta_0+\sum_{k\geq 1} \frac{1}{k!}\,
\boldsymbol{\delta}^{\oh k}(\eta_k),\
\ \ \ \mathbb{E}\|\eta\|_X^2= \|\eta_0\|_{X}^2+
\sum_{k\geq 1} \frac{1}{k!}
\|\, \|\eta_k\|_{\mathcal{U}^{\oh n}}\, \|^2_X.
\end{equation}
\end{theorem}

\begin{proof}
Using \eqref{mult-int} and Theorem \ref{th:CM},
$$
\eta=\eta_0+\sum_{k=1}^{\infty} \frac{1}{k!}\sum_{|\alpha|=k}
\eta_{\alpha}\boldsymbol{\delta}^{\oh k} \left(\frac{1}{\sqrt{\alpha!}}
E_{\alpha}\right),
$$
and we get the first equality in \eqref{eq:mult-int}
with
$$
\eta_k=\sum_{|\alpha|=k}\boldsymbol{\delta}^{\oh k}
\left(\frac{1}{\sqrt{\alpha!}}
\eta_{\alpha}\otimes E_{\alpha}\right).
$$
The second equality in \eqref{eq:mult-int} now follows from
$$
\mathbb{E}\|\eta\|_X^2=\|\eta_0\|_{X}^2+\sum_{k\geq1}
\sum_{|\alpha|=k} \|\eta_{\alpha}\|^2_X,
$$ because,
by direct computation,
\begin{equation}
\label{Enorm}
\|E_{\alpha}\|_{\mathcal{U}^{\oh n}}^2=k!\alpha!.
\end{equation}
\end{proof}

\section{The  Spaces $(\cL)_{Q,r}$}
\label{sec3}

\begin{definition}
\label{def:seq}
A sequence
$$
Q=\{q_k,\ k\geq 1\}
$$
is called a {\tt weight sequence} if $q_k\geq 1$ for all $k\geq 1$.
\end{definition}
 For $\alpha \in \cJ$ and
a real number $r$ we write
$$
q^{r\alpha}=\prod_{k\geq 1} q_k^{r\alpha_k}.
$$
Given a weight sequence $Q$,
denote by $\Lambda_{Q}$ the following
self-adjoint operator on $\mathcal{U}$:
$$
\Lambda_Q\mfku_k=q_k\mfku_k,\ k\geq 1.
$$
Then, for every $r\in \mbR$, the operator $\Lambda_Q^r$ is
defined. The domain of every $\Lambda_Q^r$ contains
finite linear combinations of $\mfku_k$ and is therefore
dense in $\mathcal{U}$. For $f$ in the domain of
$\Lambda_Q^r$ define the norm
\begin{equation}
\label{norm0}
\|f\|_{Q,r}=\|\Lambda_Q^{r}f\|_{\cU}.
\end{equation}
The operator $\Lambda_Q^r$ extends to every tensor product
 $X\otimes \cU^{\otimes n}$; we will keep the same notation for this
 extension and, in the case $X=\mbR$,
 for  the corresponding norm \eqref{norm0}.
 Denote by $(\cL)^{Q,r}(\mbF)$ the dual space of
 $(\cL)_{Q,r}(\mbF)$ relative to the inner product in $L_2(\mbF)$.
 If $\eta\in (\cL)_{Q,r}(\mbF;X)$ and
 $\zeta \in (\cL)^{Q,r}(\mbF)$, then the duality
 $$
 \langle\!\langle \eta, \zeta\rangle\!\rangle
 $$
 is defined and is an element of $X$.

 \begin{definition}
The space $(\cL)_{Q,r}(\mbF;X)$ is the closure of the
set of random elements of the form
$$
\eta=\eta_0+\sum_{k=1}^N\frac{1}{k!}\,\boldsymbol{\delta}^{\oh k}
 (\eta_k),\ N\geq 1,
$$
where $\eta_0\in X$, $\eta_k\in X\otimes \cU^{\oh k}$,
and each $\eta_k$ in the domain of $\Lambda_Q^r$, with respect to
the norm
\begin{equation}
\label{norm1}
\|\eta\|_{Q,r;X}^2=\|\eta_0\|_X^2+
\sum_{k=1}^N \frac{1}{(k!)^2}\| \, \|\eta_k\|_{Q,r}\,\|_X^2.
\end{equation}
\end{definition}

\begin{theorem}
\label{th:eqnorm}
A formal series $\eta=\sum_{\alpha\in \cJ} \eta_{\alpha} \xi_{\alpha}$
is an element of $(\cL)_{Q,r}(\mbF;X)$ if and only if
\begin{equation}
\label{eq:eqnorm}
\sum_{\alpha\in \cJ} \frac{1}{|\alpha|!}\,q^{2r\alpha} \,
\|\eta_{\alpha}\|_X^2 < \infty.
\end{equation}
The left-hand side of \eqref{eq:eqnorm}, if finite, is equal to
$\|\eta\|_{Q,r;X}^2$.
\end{theorem}

\begin{proof}
By orthonormality of $\xi_{\alpha}$,
 is enough to consider  $\eta=f\xi_{\alpha}$, $f\in X$.
 Then \eqref{eq:mult-int} implies
 $$
\|\eta\|_{Q,r;X}^2=
\frac{1}{(|\alpha|!)^2\,\alpha!}\|E_{\alpha}\|_{Q,r}^2
\|f\|_X^2=\frac{1}{|\alpha|!}
\|f\|_X^2,
$$
because the definition of $E_{\alpha}$ implies
$$
\|E_{\alpha}\|_{Q,r}^2=q^{2r\alpha}|\alpha|!\, \alpha!.
$$
\end{proof}

\begin{corollary}
\label{cor:dual}
A formal series $\zeta=\sum_{\alpha\in \cJ} \zeta_{\alpha} \xi_{\alpha}$
is an element of the dual space $(\cL)^{Q,r}(\mbF)$
if and only if
$$
\sum_{\alpha\in \cJ}
 q^{-2r\alpha} |\zeta_{\alpha}|^2<\infty.
$$
In this case,
$$
\langle\!\langle \eta, \zeta\rangle \!\rangle =
\sum_{\alpha \in \cJ}
\eta_{\alpha}\zeta_{\alpha}.
$$
\end{corollary}

For $h\in \cU$ define the stochastic exponential
\begin{equation}
\label{eq:SE}
\cE_h=\exp\left(\dot{W}(h)-\frac{1}{2}\|h\|_{\cU}^2\right).
\end{equation}

\begin{proposition}
\label{prop:SE} $\cE_h\in (\cL)^{Q,r}(\mbF)$ if and only of
$\|h\|_{Q,-r}< 1$.
\end{proposition}

\begin{proof}
We apply Corollary \ref{cor:dual}.
If $h=\sum_{k\geq 1} h_k\mfku_k$, $h_k\in \mbR$, then
by direct computation using the generating function of the
Hermite polynomials,
\begin{equation}
\label{Eh00}
\cE_h=\sum_{\alpha\in \cJ} \frac{h^{\alpha}}{\sqrt{\alpha!}}\,
 \xi_{\alpha}, \ h^{\alpha}=\prod_{k\geq 1} h_k^{\alpha_k}.
 \end{equation}
 Then, by the multinomial expansion,
 $$
\sum_{\alpha\in \cJ}
\frac{|\alpha|!}{\alpha!}h^{2\alpha}q^{-2r\alpha}=
\sum_{n\geq 0} \left(\sum_{k\geq 1} h_k^2q_k^{-r}\right)^n
=\sum_{n\geq 0} \Big(\|h\|_{Q,-r}^2\Big)^n
$$
and the result follows.
\end{proof}
In applications to evolutions equations, the number $r$ in
the solution space
$(\cL)_{Q,r;X}$ is usually {\em negative}, as we want
to help the series in \eqref{eq:eqnorm} to converge.
As a result, the following definition is appropriate.

\begin{definition}
We say the the element $h$ of $\cU$ is {\tt sufficiently
small relative to the sequence $Q$}
 if there exists a positive number $r$ such that
$$
\|h\|_{r,Q}< 1.
$$
\end{definition}

\section{The Evolution Equation and Main Result}
\label{sec4}

\begin{definition}
\label{def:Ntrip} The triple $(V,H,V^{\prime})$ of Hilbert spaces is
called \textbf{normal} if and only if

\begin{enumerate}
\item $V\hookrightarrow H \hookrightarrow V^{\prime}$
 and both embeddings
$V\hookrightarrow H $ and $H \hookrightarrow V^{\prime}$ are dense
and continuous;

\item The space $V^{\prime}$ is the dual of $V$
relative to the inner product
in $H$;

\item There exists a constant $C>0$ such that
$|(f, v)_{H}| \leq C\|v\|_{V}%
\|f\|_{V^{\prime}}$ for all $v\in V$ and $f\in H$.
\end{enumerate}
\end{definition}

For example, the Sobolev spaces $(H^{\ell+
\gamma}_{2}({\mathbb{R}}^{d}), H^{\ell}_{2}({\mathbb{R}}^{d}),
H^{\ell- \gamma}_{2}({\mathbb{R}}^{d}))$, $\gamma>0$,
$\ell\in{\mathbb{R}}$, form a normal triple.

Denote by $\langle v^{\prime}, v\rangle$, $v^{\prime}\in
V^{\prime}$, $v\in V$, the duality between $V$ and $V^{\prime}$
relative to the inner product in $H$. The properties of the normal
triple imply that $|\langle v^{\prime}, v\rangle|\leq
C\|v\|_{V}\|v^{\prime}\|_{V^{\prime}}$, and, if $v^{\prime}\in H$
and $v\in V$, then $\langle v^{\prime}, v\rangle=
(v^{\prime},v)_{H}.$

We will also use the following notation:
\begin{equation}
{\mathcal{V}(T)}=L_{2}((0,T);V),\
 {\mathcal{H}(T)}=L_{2}((0,T);H),\
 {\mathcal{V}}^{\prime}(T)=L_{2}((0,T);V^{\prime}).
  \label{sp-notation}%
\end{equation}
Given a normal triple $(V,H,V^{\prime})$,  let
${\mathbf{A}}(t):V\rightarrow V^{\prime}$ and
${\mathbf{M}}(t):V\rightarrow V^{\prime}\otimes{\mathcal{U}}$ be
bounded linear operators for every $t\in [0,T]$.

\begin{definition}
\label{def:parab}
 The solution of the stochastic evolution equation
\begin{equation}
u(t)=\inc+\int_0^t \Big(
{\mathbf{A}}(s)u(s)+f(s)+\boldsymbol{\delta}({\mathbf{M}}(s)u(s))
 \Big)ds,\ 0\leq
t\leq T,
\label{eq:evol}%
\end{equation}
with $f\in \bigcup_{r}
(\cL)_{Q,r}({\mathbb{F}};{\mathcal{V}}^{\prime}(T))$
and $u_{0}\in \bigcup_{r}
(\cL)_{Q,r} ({\mathbb{F}};H)$, is an
object with the following properties:
 \begin{enumerate}
 \item There exists a weight sequence $Q'$ and a real number
 $r'$ such that
 $$
 u\in (\cL)_{Q',r'}({\mathbb{F}};{\mathcal{V}}(T));
 $$
 \item
 For every $h\in \cU$ that is sufficiently small relative to
 the sequence $Q'$, the function
 $$
 u_h(t)=\langle\!\langle{u(t)},\cE_h\rangle\!\rangle
$$
satisfies
\begin{equation}
u_h(t)=u_h(0)+\int_{0}^{t}
\Big(\mathbf{A}(s)u_h(s)+f_h(s)+
({\mathbf{M}}(s)u_h(s),h)_{\cU}\Big)ds
\label{eq:evol-ds}%
\end{equation}
 in ${\mathcal{V}}^{\prime}(T)$.
\end{enumerate}
\end{definition}

\begin{remark}
\label{re:cont} The solution described by Definition \ref{def:parab}
belongs to the class of \textquotedblleft variational solutions",
which is quite typical for partial differential equations (see
\cite{KR,Lions,LM,Roz}, etc.) Indeed, direct computations show that
$\mathbf{D}\cE_h=h\cE_h$, and then \eqref{eq:evol-ds}
is the result of a (formal) application of the
duality $\langle\!\langle u,\cE_h\rangle \!\rangle$ to
both sides of \eqref{eq:evol}.
\end{remark}

 Fix an orthonormal basis $\mathfrak{U}=\{\mfku_k,\, k\geq 1\}$
  in ${\mathcal{U}}$. Then, for every $v\in V$, there exists a
collection $v_{k}\in V^{\prime
},\ k\geq1$, such that
\[
{\mathbf{M}}v=\sum_{k\geq1}v_{k}\otimes\mathfrak{u}_{k}.
\]
We therefore define the operators
${\mathbf{M}}_{k}:\,V\rightarrow V^{\prime}$
by setting ${\mathbf{M}}_{k}v=v_{k}$ and write
\[
{\mathbf{M}}v=\sum_{k\geq1}({\mathbf{M}}_{k}v)\otimes\mathfrak{u}_{k}.
\]
Then
\begin{equation}
\label{wp00}
\boldsymbol{\delta}(\mathbf{M}(s)u(s))=
\sum_{k\geq 1}{\mathbf{M}}_k(s)u(s)\left(  t\right)  \diamond
\xi_k,\ \xi_k=\dW(\mfku_k),
\end{equation}
and equation (\ref{eq:evol}) becomes
\begin{equation}
{u}(t)=\inc+\int_0^t\Big({\mathbf{A}}(s)u(s)+f(s)+
\sum_{k\geq 1}{\mathbf{M}}_k(s)u(s)  \diamond
\dot{W}(\mfku_k)\Big)ds.
 \label{eq:evol-b}%
\end{equation}

If the noise $\dW$ itself depends on time, this dependence
is encoded in the operator $\bM$. As a result, \eqref{eq:evol}
includes many popular evolution equations as particular cases.

\begin{example}
\label{example:main}
(1)  Let us
see how the elementary It\^{o} equation $u(t)=1+\int_0^tu(s)dw(s)$,
 where
 $w$ is a standard Brownian motion, is a
particular case of  \eqref{eq:evol}. We take $V=H=V'=\mbR$,
$\cU=L_2((0,T))$, so that
$\dW=\dot{w}(t)=\sum_{k\geq 1}
\mfku_k(t)\xi_k$ is the Gaussian white noise in time.
 Next, put $\bA=0$,  $\bM_k(t)u(t)=u(t)\mfku_k(t)$.
Then $\boldsymbol{\delta}(\bM_k(t)u(t))=u(t)\diamond\sum_{k\geq 1}
\mfku_k(t)\xi_k=u(t)\diamond \dot{w}(t)$.
The equivalence
of $u(t)=1+\int_0^t \boldsymbol{\delta}(\bM(s)u(s))ds$ and
$u(t)=1+\int_0^t u(s)dw(s)$ now follows from the
equality
$\int_0^t u(s)\diamond \dot{w}(s)ds=\int_0^tu(s)dw(s)$; see
\cite[Section 2.5]{HOUZ}.

(2) Equations with space-time white noise correspond to
$\cU=L_2((0,T))\times L_2(G)$, $G\subseteq \mbR^d$, so that
the basis in $\cU$ is naturally indexed by a pair of
indices $i,k$: $\mfku_{ik}=m_i(t)h_k(x)$.
Thus,  equation $du(t,x)=u_{xx}dt+udW(t,x)$ is a particular
case of \eqref{eq:evol} with $\bM_{ik}(t)u(t,x)=
m_i(t)h_k(x)u(t,x)$.

(3) Equations with fractional noise are  also covered
by \eqref{eq:evol}. For example, consider
$u(t)=1+\int_0^tu(s)dW^H(s)$, where $W^H$ is fractional
Brownian motion and $H>1/2$. With the appropriate
interpretation of the stochastic integral, we have
$$
\int_0^tu(s)dw^H(s)=\int_0^t u(s)\diamond \dot{w}^H(s)ds
$$
where $$
\dot{w}^H(t)=\sum_{k\geq 1}
m_k^H(t)\xi_k
$$
and $m_k^H$ are suitable elements of $L_2((0,T))$ (see
for example \cite[Chapter 5]{Nualart}. Then, in \eqref{eq:evol},
 we take
$\cU=L_2((0,T))$ and $\bM_k(t)u(t)=m_k^H(t)u(t)$.
\end{example}

Let us fix an orthonormal basis $\mathfrak{U}$ in $\mathcal{U}$ so that
equation (\ref{eq:evol}) becomes (\ref{eq:evol-b}). With
$\xi_k=\dot{W}(\mathfrak{u}_k)$, define $\xi_{\alpha}, \ \alpha\in \cJ$
according to \eqref{eq:basis}. Given a Hilbert space $X$,
 every element $\eta$ of
$(\cL)_{Q,r}(\mathbb{F};X)$ can be written as a formal series
$\sum_{\alpha}\eta_{\alpha}\xi_{\alpha}$ with $\eta_{\alpha}\in X$
satisfying \eqref{eq:eqnorm}. The following theorem provides a
necessary and sufficient condition for the formal series
$\sum_{\alpha}u_{\alpha}(t)\xi_{\alpha}$ to be a solution of
\eqref{eq:evol-b}.

\begin{theorem}
\label{th: equivalence1}
Let $u=\sum_{\alpha\in{\mathcal{J}}}u_{\alpha}%
\xi_{\alpha}$ be an element of
$(\cL)_{Q',r'}({\mathbb{F}};{\mathcal{V}}(T))$.
The process $u$ is a solution of equation (\ref{eq:evol-b}) if
and only if the functions $u_{\alpha}$ have the following
properties:

\begin{enumerate}
\item every $u_{\alpha}$ is an element of ${\mathbf{C}} \left(
[0,T];H)\right).  $
\item the system of equalities
\begin{equation}
u_{\alpha}(t)=\inc_{\alpha}+\int_{0}^{t}\left(  \
{\mathbf{A}}(s)u_{\alpha}(s)+f_{\alpha}(s)
+\sum_{k\geq1}\sqrt{\alpha_{k}}\,{\mathbf{M}}_{k}(s)%
u_{\alpha-\varepsilon_{k}}(s)\right)  ds
\label{eq:evol-S}%
\end{equation}
holds in $\cV^{\prime}(T)$ for all
$\alpha\in{\mathcal{J}}$.
\end{enumerate}
\end{theorem}

\begin{proof}
If $h=\sum_{k}h_k\mfku_k$, then,
by \eqref{Eh00} and Corollary \ref{cor:dual}
$u_h=\sum_{\alpha\in \cJ}\frac{u_{\alpha}h^{\alpha}}{\sqrt{\alpha!}}.$
 With no loss of generality we can assume that
$\sum_{k\geq 1} (q'_k)^2<\infty$.
Then, by a general result  from functional analysis \cite{Kond},
 the mapping $u_h: \cU\mapsto \mathcal{V}(T)$ is analytic at zero.
  By direct computation,
 \begin{enumerate}
 \item if $u_h$ satisfies \eqref{eq:evol-ds}, then
 $$
 u_{\alpha}=\frac{1}{\sqrt{\alpha!}}
 \,\frac{\partial^{|\alpha|}u_h}
 {\partial h_1^{\alpha_1} h_2^{\alpha_2}\ldots}\Big{|}_{h=0}
 $$
 and each $u_{\alpha}$ satisfies \eqref{eq:evol-S}.
 \item if $u_{\alpha}$ satisfies \eqref{eq:evol-S} and
 $u_h=\sum_{\alpha} u_{\alpha}h^{\alpha}$, then $u_h$
 satisfies \eqref{eq:evol-ds}.
 \end{enumerate}
 \end{proof}

 This simple but very helpful result establishes the equivalence of the
\textquotedblleft physical" (\ref{eq:evol-b}) and the (stochastic) Fourier
(\ref{eq:evol-S}) forms of equation (\ref{eq:evol}). The system of equations
(\ref{eq:evol-S}) is often referred in the literature as the \textit{
propagator} of equation (\ref{eq:evol-b}). Note that the propagator is
lower-triangular and can be solved by induction on $|\alpha|$.

To prove existence and uniqueness of solution of (\ref{eq:evol}),
 we make the
following assumptions about the operators $\bA$ and $\bM$.

(A):\textbf{ }\textit{For every }$U_{0}\in H$\textit{ and }$F\in
\mathcal{V}^{\prime}(T)$\textit{, there exists
a unique
function }$U\in\mathcal{V}$\textit{ that solves
the deterministic equation }%
\begin{equation}
U(t)=U_0+{\mathbf{A}}(t)U(t)+F(t),
\label{eq: determ}%
\end{equation}
\textit{ and there exists a constant }$C=C\left(
{\mathbf{A}},T\right)
$\textit{ so that }%
\begin{equation}
\Vert U\Vert_{{\mathcal{V}}(T)}\leq C({\mathbf{A}},T)\big(\Vert
U_{0}\Vert
_{H}+\Vert F\Vert_{{\mathcal{V}}^{\prime}(T)}\big).
 \label{det-reg}%
\end{equation}
In particular, the operator ${\mathbf{A}}$ generates a semi-group
$\Phi=\Phi_{t,s}, \, t\geq s\geq0,$ and
$$
U(t)=\Phi_{t,0}U_0+\int_0^t\Phi_{t,s}F(s)ds.
$$

(M): For every $v\in{\mathcal{V}}(T)$ and $k\geq 1$,
\begin{equation}
\int_{0}^{T}\left\Vert \int_{0}^{t}\
\Phi_{t,s}{\mathbf{M}}_{k}(s)v\left( s\right)  ds\right\Vert
_{V}^{2}dt\leq C_{k}^{2}\left\Vert v\right\Vert
_{\mathcal{V}(T)}^{2}, \label{Ck-def}%
\end{equation}
with numbers $C_{k}$ independent of $v$.

\begin{remark}
\label{rm:evol} There are various types of assumptions on the
operator ${\mathbf{A}}$ that yield the statement of the assumption
(A). In particular, (A) holds if the operator ${\mathbf{A}}$ \ is
coercive in $\left( V,H,V^{\prime}\right)  $:
\[
\langle{\mathbf{A}}(t)v,v\rangle+\gamma\Vert v\Vert_{V}^{2}
\leq C\Vert v\Vert_{H}^{2}%
\]
for every $v\in V$, $t\in [0,T]$,
 where $\gamma>0$ and $C\in{\mathbb{R}}$ are both
independent of $v,t$.
\end{remark}

The following theorem is the main result of this paper.

\begin{theorem}
\label{th:main}
In addition to (A) and (M), assume that,
for some positive number $r$ and a weight
sequence $Q$,
$\inc\in (\mathcal{L})_{Q,-r}(\mathbb{F};H)$
 and $f\in(\cL)_{Q,-r}(\mathbb{W};\cV'(T))$.
If
$$
\sum_{\alpha \in \cJ} \frac{1}{q^{2r\alpha}\,\alpha!}<\infty,
$$
then  there exists a weight  sequence $Q^{\circ}$
and a negative real number $r^{\circ}$
such that equation \eqref{eq:evol} has a unique
solution
$u\in (\mathfrak{L})_{Q^{\circ},r^{\circ}}(\mathbb{F};\cV(T))$ and
\begin{equation}
  \label{eq:main1}
\|u\|_{Q^{\circ},r^{\circ};\cV(T)}^2
\leq C\cdot \left( \|\inc\|_{Q,r;H}^2
+\|f\|_{Q,r;\cV'(T)}^2\right).
\end{equation}
The number $C>0$ depends only on $Q,r,T$
and the operators $\bA, \bM$.
\end{theorem}

\begin{proof} The proof consists of two steps: first, we prove the
result for deterministic functions $\inc,f$ and then use linearity to
extend the result to the general case.

\emph{Step 1.} Assume that the functions $\inc\in H$,
$f\in \cV'(T)$ are deterministic so that $\inc_h=\inc$,
$f_h=f$.
By assumptions \eqref{det-reg} and \eqref{Ck-def}
the evolution equation
\begin{equation}
\label{eq:Str-sol}
u_h(t)=\inc+\int_0^t\left(\bA(s)u_h(s)+f(s)+\sum_{k=1}^{\infty}
\bM_k(s)u_h(s)h_k\right)\ds
\end{equation}
has a unique solution in $\cV(T)$ as long as
$\sum_{k\geq 1} h_k^2$ is small enough; since the equation
is linear, the  solution
is an analytic function of $h_k$.

Next,
\begin{equation}
u(t)=\sum_{\alpha\in \mathcal{J}} u_{\alpha}(t)\xi_{\alpha}
\end{equation}
and,  by \eqref{eq:evol-S}, the coefficients $u_{\alpha}$ satisfy
\begin{equation}
\begin{split}  \label{eq:S-syst-G}
u_{(0)}(t)&=\inc
+\int_0^t(\mathbf{A}(s) u_{(0)}(s)+f(s))ds,\ \ |\alpha|=0; \\
u_{\epsilon_k}(t )&
= \int_0^t\mathbf{A}(s) u_{\epsilon_k}(s)ds+ \int_0^t
\mathbf{M}_k(s) u_{(0)}(s )ds, \ \ |\alpha|=1;\\
u_{\alpha}(s )&= \int_0^t \mathbf{A}(s) u_{\alpha}(s )ds+
\sum_{k} \sqrt{\alpha_k} \int_0^t
 \mathbf{M}_k(s) u_{\alpha-\epsilon_k}(s )ds,\
|\alpha|>1.
\end{split}%
\end{equation}
Denote by $\Phi=\Phi_{s,t},\ t\geq s\geq 0$ the semigroup generated
by the operator $\mathbf{A}(t)$. It follows by induction on $|\alpha|$
that
\begin{equation}
\begin{split}
u_{(0)}(t )& = \Phi_{t,0}\inc(x) +\int_0^t\Phi_{t,s}f(s)ds, \
|\a|=0;\\
u_{\epsilon_k}(t )&= \int_0^t \Phi_{t,s}\mathbf{M}_k(s)
u_{(0)}(s )ds,\ |\a|=1; \\
u_{\alpha}(t )&=\frac{1}{\sqrt{\alpha!}} \int_0^t\int_0^{s_n}
\ldots
\int_0^{s_2} \Phi_{t,s_{n}}
\mathbf{M}_{i_n}\Phi_{s_{n},s_{n-1}} \cdots \mathbf{M}_{i_{n-1}}%
\Phi_{s_2,s_1}\mathbf{M}_{i_1} u_{(0)}\,
ds_1\ldots ds_n,\\
& \ |\alpha|=n>1, K_{\alpha}=\{i_1, \ldots, i_n\}.
\end{split}%
\end{equation}
By assumptions \eqref{det-reg} and \eqref{Ck-def},
\begin{equation}
\label{ual-est}
\|u_{\alpha}(t)\|_{\cV(T)}^2
 \leq \frac{(|\alpha|!)^2}{\alpha!}
\left(\|\inc\|_{H}^2
+  \|f\|^2_{\cV'(T)} \right)\prod_{k\geq 1} C_k^{2\alpha_k}.
\end{equation}
It is known (see, for example, \cite[pages 31, 35]{HOUZ}) that
\begin{equation}
|\alpha|!\leq\alpha!(2{\mathbb{N}})^{2\alpha}, \qquad
\sum_{\alpha\in \cJ}(2\mathbb{N})^{-r\alpha}<\infty
\ {\rm if\ and \ only \ if\ } r>1;
 \label{HOUZ-ineq}%
\end{equation}
where
\[
(2{\mathbb{N}})^{2\alpha}=2^{2|\alpha|}\,\prod_{k\geq1}k^{2\alpha_{k}}.
\]
Define the sequence $Q^{\circ}=\{q_k^{\circ},\, k\geq 1\}$ by
\begin{equation}
\label{bw}
q_k^{\circ}=2k(1+C_k).
\end{equation}
Then \eqref{eq:eqnorm}  and \eqref{ual-est} imply
$$
\|u\|_{Q^{\circ},r^{\circ};\cV(T)}^2
\leq C\cdot \left( \|\inc\|_{H}^2
+\|f(t)\|_{\cV'(T)}^2\right)\ {\rm for \ every\ } r^{\circ}<-2.
$$

\emph{Step 2.} As in Step 1, existence and uniqueness of solution
follows from unique solvability of the parabolic equation
\eqref{eq:Str-sol}, and it remains to establish (\ref{eq:main1}).

Given  $v\in H$, $F\in \cV'(T)$, and $\gamma\in \cJ$,
denote by $u(t ;v,F,\gamma)$
the solution of %
\eqref{eq:evol} with $\inc=v\xi_{\gamma}$, $f=F\xi_{\gamma}$.
If $\inc=\sum_{\alpha\in \mathcal{J}}\inc_{\alpha} \xi_{\alpha}$,
$f=\sum_{\alpha\in \cJ}f_{\alpha}\xi_{\alpha}$,
then, by linearity,
\begin{equation}  \label{eq:sol-sum}
u(t )=\sum_{\gamma\in \mathcal{J}}
u(t ;\inc_{\gamma},f_{\gamma},\gamma).
\end{equation}
It follows from (\ref{eq:S-syst-G}) that
$u_{\alpha}(t ;v,F,\gamma)=0$ if $|\alpha|<|\gamma|$ and
\begin{equation}
\frac{u_{\alpha+\gamma}(t ;v,F,\gamma)} {\sqrt{(\alpha+\gamma)!}}
= \frac{%
u_{\alpha}\left(t ;\frac{v}{\sqrt{\gamma!}},
\frac{F}{\sqrt{\gamma!}},(0)\right)}{\sqrt{\alpha!}}.
\end{equation}
Using the results of Step 1,
\begin{equation}
\|u(t;\inc_{\gamma},
 f_{\gamma},\gamma)\|_{Q^{\circ},r^{\circ};\cV(T)}
\leq \frac{C}{\sqrt{\gamma!}} \left(\|v_{\gamma}\|_{H}
+  \|f_{\gamma}(t)\|_{\cV'(T)}
 \right).
\end{equation}
With  $Q^{\circ}$ defined in \eqref{bw}  and $r^{\circ}<-2$,
 \eqref{eq:main1} now follows from \eqref{eq:sol-sum} by the triangle
and the Cauchy-Schwartz inequalities.
\end{proof}

\begin{corollary}[A generalization of the
Krylov-Veretennikov formula from \cite{KV}]
Assume that $\inc$ and $f$ are deterministic and
take $q^{\circ}_k$ from  \eqref{bw}. Define the
sequence $U_n=U_n(t)$, $n\geq 0$ by induction
\begin{equation}
U_0(t)=u_{(0)}(t),\ U_{n+1}(t)=\int_0^t
\Phi_{t,s}\boldsymbol{\delta}(\bM^{\circ}(s)U_n(s)) ds,\ n>0,
\end{equation}
where $\bM^{\circ}=(q^{\circ}_1\bM_1,q^{\circ}_2\bM_2,\ldots)$.
Then
$$
\sum_{|\alpha|=n}(q^{\circ})^{\alpha}u_{\alpha}(t)\,\xi_{\alpha}
=U_n(t).
$$
\end{corollary}

\begin{proof}
By \eqref{eq:S-syst-G},
\[
(q^{\circ})^{\alpha}u_{\alpha}(t)
=\int_{0}^{t}{\mathbf{A}}(q^{\circ})^{\alpha}u_{\alpha}%
(s)ds+\sum_{k\geq1}\int_{0}^{t}q_{k}\sqrt{\alpha_{k}}\,{\mathbf{M}}%
_{k}(q^{\circ})^{\alpha-\varepsilon_{k}}
u_{\alpha-\varepsilon_{k}}(s)ds.
\]
Therefore,
$$
(U_{n}(t))_{\alpha} =
\sum_{k\geq1}\sqrt{\alpha_{k}}\,\int_{0}^{t}\Phi_{t-s}
\bM^{\circ}_{k}(U_{n-1}(s))_{\alpha-\varepsilon_{k}}ds.
$$
By \eqref{eq:IBasG},
$$
U_n(t)=\int_0^t\Phi_{t,s}
\boldsymbol{\delta}(\bM^{\circ}(s)U_{n-1}(s)) ds
$$
and the result follows.
\end{proof}

\begin{remark}
The main goal of Theorem \ref{th:main} is universality:
to cover the largest possible class of equations. It is
therefore inevitable, and natural, that, for many particular
equations, there are much better regularity results than
\eqref{eq:main1}.
\end{remark}


\providecommand{\bysame}{\leavevmode\hbox to3em{\hrulefill}\thinspace}
\providecommand{\MR}{\relax\ifhmode\unskip\space\fi MR }
\providecommand{\MRhref}[2]{%
  \href{http://www.ams.org/mathscinet-getitem?mr=#1}{#2}
}
\providecommand{\href}[2]{#2}

\end{document}